\documentclass[preprint,review, 10pt,3p,times]{elsarticle}

\usepackage{amssymb,amsthm,amsmath,amsfonts,graphicx}
\usepackage[english]{babel}
\usepackage{latexsym}
\usepackage{epsfig}
\usepackage{url}
\usepackage[T1]{fontenc}
\usepackage{color}
\usepackage{enumerate}
\newcommand{\bc}{\begin{center}}
\newcommand{\ec}{\end{center}}
\newtheorem{theorem}{Theorem}
\newtheorem{proposition}{Proposition}
\newtheorem{corollary}{Corollary}
\newtheorem{lemma}{Lemma}

\newtheorem{defi}{Definition}
\begin{document}

\begin{frontmatter}

\title{On global location-domination in graphs}

\author[upc]{C.~Hernando}
\ead{carmen.hernando@upc.edu}

\author[upc]{M.~Mora}
\ead{merce.mora@upc.edu}

\author[upc]{I. M.~Pelayo\corref{cor1}}
\ead{ignacio.m.pelayo@upc.edu}

\cortext[cor1]{Corresponding author}

\address[upc]{Universitat Politècnica de Catalunya, Barcelona, Spain}

\author{}

\address{}

%%%%%%%%%%%%%%%%%%%%%%%%%%%%%%%%%%%%%%%%%%%%%%%%%%%%%%%%%%%%%%%%%%%%
%%%%%%%%%%%%%%%%%%%%%%%%%%%%%%%%%%%%%%%%%%%%%%%%%%%%%%%%%%%%%%%%%%%%
\begin{abstract}
\small 
A dominating set $S$  of a graph $G$ is called \emph{locating-dominating}, \emph{LD-set} for short, if every vertex $v$ not in  $S$ is uniquely determined by the set of neighbors of $v$ belonging to $S$.
Locating-dominating sets of minimum cardinality are called $LD$-codes and the cardinality of an LD-code is the \emph{location-domination number} $\lambda(G)$.
An LD-set $S$ of a graph $G$ is  \emph{global} if it is an LD-set of both $G$ and its complement $\overline{G}$.
The   \emph{global location-domination number} $\lambda_g(G)$ is the minimum cardinality of a global LD-set of $G$.
In this work, we give some relations between  locating-dominating sets and the location-domination number in a graph and its complement.
\end{abstract}


\end{frontmatter}

%%%%%%%%%%%%%%%%%%%%%%%%%%%%%%%%%%%%%%%%%%%%%%%%%%%%%
%%%%%%%%%%%%%%%%%%%%%%%%%%%%%%%%%%%%%%%%%%%%%%%%%%%%%
%%%%%%%%%%%%%%%%%%%%%%%%%%%%%%%%%%%%%%%%%%%%%%%%%%%%%
%%%%%%%%%%%%%%%%%%%%%%%%%%%%%%%%%%%%%%%%%%%%%%%%%%%%%
%%%%%%%%%%%%%%%%%%%%%%%%%%%%%%%%%%%%%%%%%%%%%%%%%%%%%
%%%%%%%%%%%%%%%%%%%%%%%%%%%%%%%%%%%%%%%%%%%%%%%%%%%%%
\section{Introduction}

Many problems involving detection devices can be modeled with graphs.
Detection devices and the objects or intruders to be detected occupy some vertices of a graph.
We are interested in finding the minimum number of devices needed according to the type of devices and the necessity of locating the intruder.
This gives rise to consider locating and dominating sets.
Locating-dominating sets can be used to determine the location of an object in a graph if devices can detect only objects in its neighborhood and the object cannot occupy the same vertex as detection devices.

Let $G=(V,E)$ be a simple, not necessarily connected, finite graph.
The \emph{open neighborhood} of a vertex $v\in V$ is $N_G(v)=\{u\in V : uv\in E\}$ and the \emph{close neighborhood} is $N_G[v]=\{u\in V : uv\in E\}\cup \{ v \}$.
The \emph{complement} of a graph $G$,  denoted by $\overline{G}$, is the  graph on the same vertices such that two vertices are adjacent in $\overline{G}$ if and only if they are not adjacent in $G$.
The distance between vertices $v,w\in V$ is denoted by $d_G(v,w)$.
We write $N(u)$ or $d(v,w)$ if the graph G is clear from the context.
Assume that $G$ and $H$ is a pair of graphs whose vertex sets are disjoint.
The \emph{union} $G+ H$ is the graph with vertex set $V(G)\cup V(H)$ and edge set $E(G)\cup E(H)$.
The \emph{join} $G\vee H$ has $V(G)\cup V(H)$ as vertex set and $E(G)\cup E(H)\cup \{ u v : u\in v(G)\textrm{ and }v\in V(H)\}$ as edge set.
For further notation, see \cite{chlezh11}.

A set $D\subseteq V$ is a \emph{dominating set} if for every vertex $v\in V\setminus D$, $N(v)\cap D\neq\emptyset$.
The \emph{domination number} of $G$, denoted by $\gamma(G)$, is the minimum cardinality of a dominating set of $G$.
A dominating set is \emph{global} if it is a dominating set of both $G$ and its complement graph, $\overline{G}$.
The minimum cardinality of a global dominating set of $G$ is the \emph{global domination number} of $G$, denoted with $\gamma_g (G)$ \cite{brca98,brdu90,sam89}.
If $D$ is a subset of $V$ and $v\in V\setminus D$, we say that  $v$ \emph{dominates}  $D$ if $D\subseteq N(v)$.

A set $S\subseteq V$  is  a \emph{locating set}  if every vertex is uniquely determined by its vector of distances to the vertices in $S$.
The \emph{location number} of $G$  $\beta(G)$ is the minimum cardinality of a locating set of $G$~\cite{hame,ours2,slater75}.

A set $S\subseteq V$ is  a \emph{locating-dominating set}, \emph{LD-set} for short,  if $S$ is a dominating set such that for every two different vertices  $u,v\in V\setminus S$, $N(u)\cap S\neq N(v)\cap S.$
The \emph{location-domination number} of $G$,   denoted by $\lambda (G)$, is the  minimum cardinality of a locating-dominating set.
A locating-dominating set of cardinality $\lambda(G)$ is called an \emph{LD-code}~\cite{slater88}.
Certainly,  every LD-set of a non-connected graph $G$ is  the  union of LD-sets of its connected components and the location-domination number is the sum of the location-domination number of its connected components.
Notice also that a locating-dominating set is both a locating set and a dominating set, and thus $\beta (G)\le \lambda (G)$ and  $\gamma(G)\le \lambda (G)$. 
LD-codes and the location-domination parameter have been intensively studied during the last decade; 
see \cite{bchl,bcmms07,cahemopepu12,clm11,ours3}
A complete and regularly updated list of papers on locating dominating codes is to be found in \cite{lobstein}.

A \emph{block} of a graph is a maximal connected subgraph with no cut vertices.
%An \emph{unicyclic} graph is a connected graph with the same number of vertices and edges, i.e., such that it contains exactly one cycle.
A graph is a \emph{block graph} if it is connected and each of its blocks is complete.
A connected graph $G$ is a \emph{cactus} if all its blocks are cycles or complete graphs of order at most 2.
Cactus are characterized as those graphs such that two different cycles share at most one vertex.
A \emph{block-cactus} is a connected graph such that each of its blocks is either a  cycle or a complete graph.
The family of block-cactus graphs is interesting because, among other reasons,  it contains all cycles, trees, complete graphs, block graphs, unicyclic graphs and cactus (see Figure \ref{fig.blockcactus}).
Cactus, block graphs, and block-cactus  have been studied  extensively in  different contexts, including the domination one; see \cite{ch06,hs12,ravo98,xsz06,zv98}.

\vspace{.05cm}
%%%%%%%%%%%%%%%%  pendent: afegir uniciclics a la figura!!!!
\begin{figure}[hbt]
\begin{center}
\includegraphics[height=5.2cm]{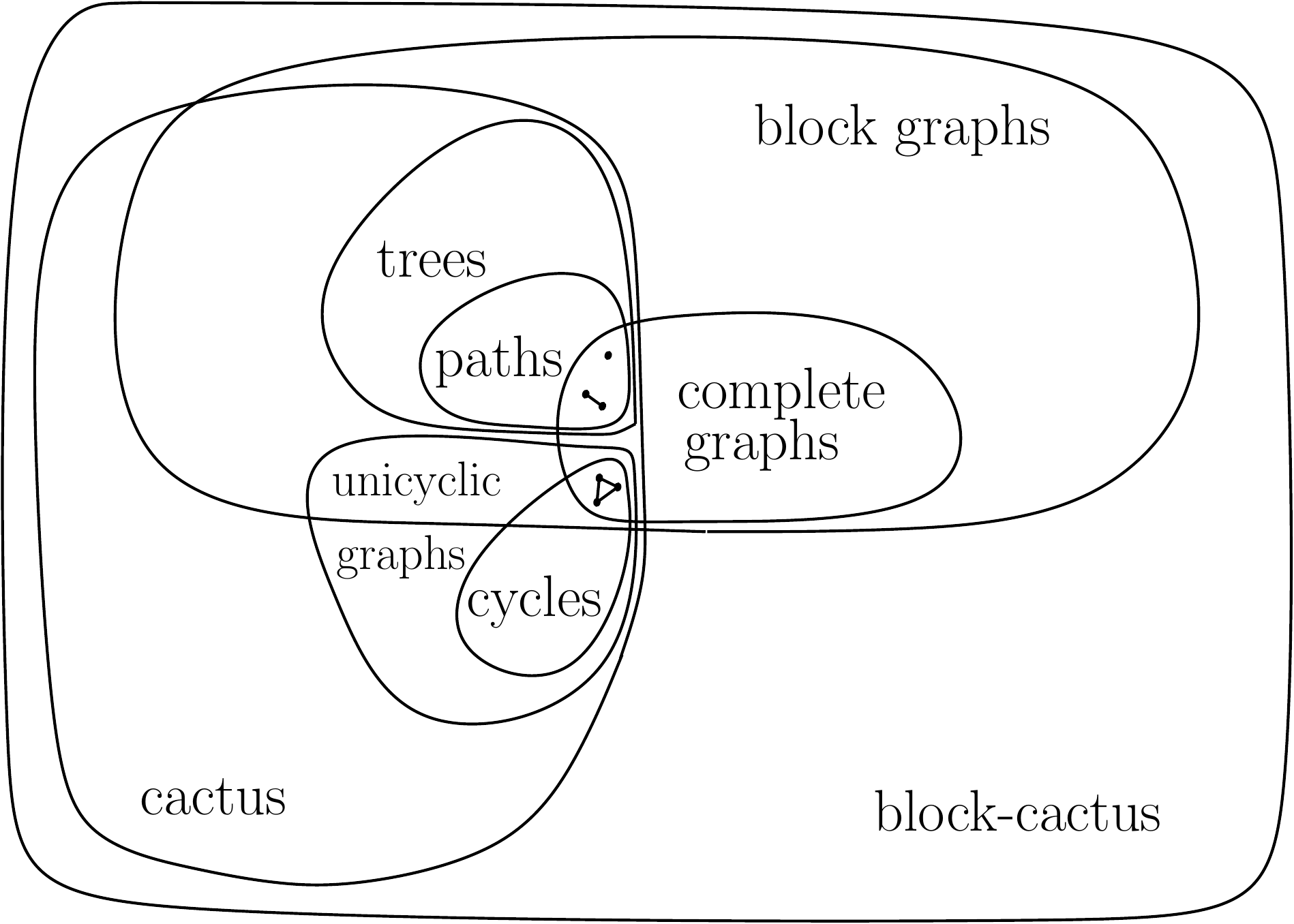}
\caption{Families of block-cactus.}\label{fig.blockcactus}
\end{center}
\end{figure}
%%%%%%%%%%%%%%%%%
%%%%%%%%%%%%%%%%%

The remaining part of this paper is organized as follows. 
In Section 2,  we deal with the problem of relating the  locating-dominating sets and the location-domination number of a graph and its complement. 
Also,  \emph{global LD-sets} and   \emph{global LD-codes} are defined.
In Section 3, we introduce the so-called  \emph{global location-domination number}, and show some basic properties for this new parameter.
In Section 4,  we are concerned with the study  of the sets and parameters  considered in the preceding sections for the  family  of  \emph{block-cactus} graphs.
Finally, the last section is devoted to address some open problems.

%%%%%%%%%%%%%%%%%%%%%%%%%%%%%%%%%%%%%%%%%%%%%%%%%%%%%
%%%%%%%%%%%%%%%%%%%%%%%%%%%%%%%%%%%%%%%%%%%%%%%%%%%%%
%%%%%%%%%%%%%%%%%%%%%%%%%%%%%%%%%%%%%%%%%%%%%%%%%%%%%
%%%%%%%%%%%%%%%%%%%%%%%%%%%%%%%%%%%%%%%%%%%%%%%%%%%%%
%%%%%%%%%%%%%%%%%%%%%%%%%%%%%%%%%%%%%%%%%%%%%%%%%%%%%
%%%%%%%%%%%%%%%%%%%%%%%%%%%%%%%%%%%%%%%%%%%%%%%%%%%%%
\section{Relating   $\lambda (G)$ to $\lambda (\overline{G})$}

This section is devoted to approach the relationship between $\lambda (G)$ and $\lambda (\overline{G})$, for any arbitrary graph $G$.
Some of the results we present were previously shown  in \cite{ours3} and we include them for the sake of completeness.

%%%%%%%%%%%%%%%%%%%%%%%%%%%%%%% els conjunts $N_G(u) \cap S$ són diferents si i nomes si $N_{G compl} \cap S$ son diferents

Notice that $N_{\overline{G}}(x)\cap S = S\setminus N_{G}(x)$ for any set $S\subseteq V$ and any vertex $x\in V\setminus S$. 
A straightforward consequence of this fact is the following lemma.

\begin{lemma} \label{lem.vecinosdistintos}
Let $G=(V,E)$ be a graph and $S\subseteq V$. If $x,y\in V\setminus S$, then
$N_G(x)\cap S \not= N_G(y)\cap S$ if and only if $N_{\overline{G}}(x)\cap S \not= N_{\overline{G}}(y)\cap S$.
\end{lemma}

As an immediate consequence of this lemma, the following result is derived.

\begin{proposition} \label{pro.domi}
If $S\subseteq V$ is an LD-set of a graph $G=(V,E)$, then $S$ is an LD-set of $\overline{G}$ if and only if $S$ is a dominating set of $\overline{G}$.
\end{proposition}

\begin{proposition} [\cite{ours3}]\label{pro.vertexdom}
If $S\subseteq V$ is an LD-set of a graph $G=(V,E)$, then $S$ is an LD-set of $\overline{G}$ if and only if there is no vertex in $V\setminus S$ dominating $S$ in $G$.
%Moreover, if there is a vertex $u\in V\setminus S$ dominating $S$, this is unique and then $S\cup \{ u \}$ is an LD-set of $\overline{G}$.
\end{proposition}
\begin{proof}
By Proposition \ref{pro.domi}, $S$ is an LD-set of $\overline{G}$ if and only if $S$ is a dominating set of $\overline{G}$. But $S$ is a dominating set of $\overline{G}$ if and only if $N_{\overline{G}}(u)\cap S\not= \emptyset$, for any vertex $u\in V\setminus S$. This condition is equivalent to $N_G(u)\cap S\not= S$ for any vertex $u\in V\setminus S$. Therefore, $S$ is an LD-set of $\overline{G}$ if and only if there is no vertex $u\in V\setminus S$ such that $S\subseteq N_G(u)$, that is, there is no vertex in $V\setminus S$ dominating $S$.
\end{proof}
%%%%%%%%%%%%%%%%%%%%%%%%%%%%%%%%%%%%%%%%%%%%%%%%%%%%%

%%%%%% si un LD set S de G no lo es del complementario, añadimos el unico vertice que domina S para tener un  LD set del complementario

\begin{proposition} [\cite{ours3}]\label{pro.union}
If $S\subseteq V$ is an LD-set of a graph $G=(V,E)$ then there is at most one vertex $u\in V\setminus S$ dominating $S$, and in the case it exists, $S\cup \{ u \}$ is an LD-set of $\overline{G}$.
\end{proposition}
\begin{proof}
By definition of LD-set of $G$, there is at most one vertex adjacent to all vertices of $S$. Moreover, $u$ is the only vertex not adjacent to any vertex of $S$ in $\overline{G}$. Therefore $S\cup \{ u \}$ is an LD-set of $G$ and a dominating set of $\overline{G}$. By Proposition \ref{pro.domi}, it is also an LD-set of $\overline{G}$.
\end{proof}
%%%%%%%%%%%%%%%%%%%%%%%%%%%%%%%%%%%%%%%%%%%%%%%%%%%%%

%%%%%%%%%%%%%%%%%%%%%%%%%%%%%%%%%%%%%%%%%%%%%%%%%%%%%
\begin{theorem} [\cite{ours3}]\label{cor.difuno}
For every graph $G$, $|\lambda (G) -\lambda (\overline{G})|\le 1$.
\end{theorem}
\begin{proof}
If $S$ has an LD-code of $G$ not containing a vertex dominating $S$, then $S$ is an LD-set of $\overline{G}$ by \ref{pro.vertexdom}. Consequently, $\lambda (\overline{G}) \le \lambda (G)$.
If $S$ is an LD-code of $G$ with a vertex $u\in V\setminus S$ dominating $S$, then $S\cup \{ u \}$ is an LD-set of $\overline{G}$ by \ref{pro.union}. 
Consequently, $\lambda (\overline{G}) \le \lambda (G)+1$.
In any case, $\lambda (\overline{G}) -\lambda (G)\le 1$.
By symmetry, $\lambda (G) - \lambda (\overline{G})\le 1$ and, hence, $|\lambda (G) -\lambda (\overline{G})|\le 1$.
\end{proof}
%%%%%%%%%%%%%%%%%%%%%%%%%%%%%%%%%%%%%%%%%%%%%%%%%%%%%

According to the preceding result, for every graph $G$,
$\lambda (\overline{G})\in\{\lambda (G)-1,\lambda (G),\lambda (G)+1\}$, all cases being  feasible for some connected graph $G$.
For example, it is easy to see that the complete graph $K_n$ of order $n\ge 2$ satisfy $\lambda (\overline{K_n})=\lambda (K_n)+1$, the star $K_{1,n-1}$ of order $n\ge 2$ satisfies
$\lambda (\overline{K_{1,n-1}})=\lambda (K_{1,n-1})$, and the bi-star $K_{2}(r,s)$, $r,s\ge 2$, obtained by joining the central vertices of two stars $K_{1,r}$ and $K_{1,s}$ , satisfies
$\lambda (K_{2}(r,s))=\lambda (\overline{K_{2}(r,s)})+1$.

%(redactar mejor!)
We intend to obtain  either necessary or sufficient conditions for a graph $G$ to satisfy $\lambda (\overline{G})>\lambda (G)$, i.e., $\lambda (\overline{G})= \lambda (G) +1$.  
After noticing that this fact is closely related to the existence or not  of sets that are simultaneously locating-dominating sets in both $G$ and its complement  $\overline{G}$, the following definition is introduced.

\begin{defi}\rm
A set $S$ of vertices of a graph $G$  is a \emph{global LD-set} if $S$ is an LD-set of both $G$ and its complement $\overline{G}$. 
\end{defi}

Certainly, an LD-set is non-global if and only if there exists a (unique)  vertex $u\in V(G)\setminus S$ which dominates $S$, i.e., such that $S\subseteq N(u)$.

Accordingly, an LD-code $S$ of a graph $G$ is said to be  \emph{global} if it is a global LD-set, i.e. if $S$ is both an LD-code of $G$ and an LD-set of $\overline{G}$.
In terms of this new definition, a significant  result proved in \cite{ours3} can be presented as follows.

%%%%%%%%%%%%%%%%%%%%%%%%%%%%%%%%%%%%%%%%%%%%%%%%%%%%%
\begin{proposition} [\cite{ours3}]\label{pro.globalImplica}
If $G$ is a graph with a global LD-code, then $\lambda (\overline{G})\le \lambda (G)$.
\end{proposition}

%%%%%%%%%%%%%%%%%%%%%%%%%%%%%%%%%%%%%%%%%%%%%%%%%%%%%

%%%%%%%%%%%%%%%%%%%%%%% como son los grafos que tienen un LD-set no global %%%%%%%%%%%%%%%%%%%%%%%%%%%%%%
\begin{proposition}\label{andreu}
If $G$ is a graph with a non-global LD-set $S$ and $u$ is the only vertex dominating $S$, then the following conditions are satisfied:
\begin{enumerate}[(i)]
  \item The eccentricity of $u$ is $ecc(u)\le 2$;
  \item the radius of $G$ is $rad (G)\le 2$;
  \item the diameter of $G$ is $diam (G)\le 4$;
  \item the maximum degree of $G$ is $\Delta (G)\ge \lambda (G)$.
\end{enumerate}
\end{proposition}
\begin{proof}
If $x\in N(u)$, then $d(u,x)=1$.  If $x\notin N(u)$, since $S$ is a dominating set of $G$, then there exists a vertex $y\in S\cap N(x) \subseteq N(u)$.  Hence, $ecc(u)\le 2$. Consequently,  $rad (G) \le 2$ and $diam (G)\le 4$.  By the other hand,
$deg_G(u)=|N_G(u)|\ge |S|=\lambda (G)$, implying that $\Delta (G)\ge \lambda (G)$.
\end{proof}
%%%%%%%%%%%%%%%%%%%%%%%%%%%%%%%%%%%%%%%%%%%%%%%%%%%%%

%%%%%%%%%%%%%%%%%%%%%%%%%%%%%%%%%%%%%%%%%%%%%%%%%%%%%
\begin{corollary}
If $G$ is a graph satisfying $\lambda (\overline{G}) =\lambda (G)  + 1$, then $G$ is a connected graph such that $rad (G)\le 2$, $diam (G)\le 4$ and $\Delta (G)\ge \lambda (G)$.
\end{corollary}
%%%%%%%%%%%%%%%%%%%%%%%%%%%%%%%%%%%%%%%%%%%%%%%%%%%%%

\begin{figure}[hbt]
  \begin{center}
        \includegraphics[width=.18\textwidth]{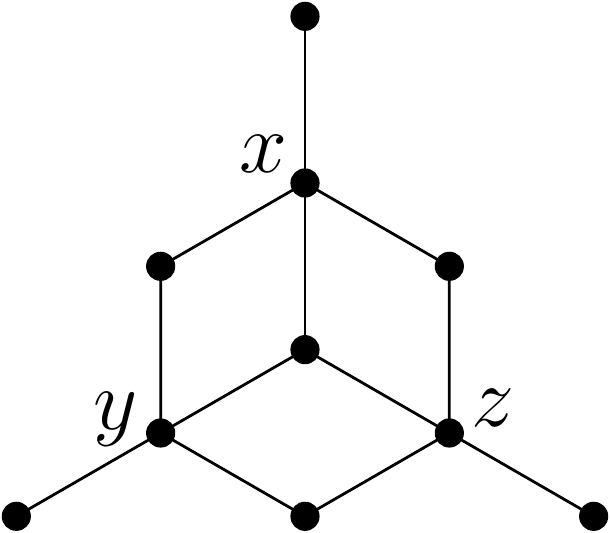}
  \end{center}
  \caption{This graph satisfies: $rad(G)=2$, $diam(G)=4$, $\lambda(G)=3$, $\lambda(\overline G)=4$ and $\{x,y,z,\}$ is a non-global LD-code.}
  \label{hexagons}
\end{figure}

The above result is tight in the sense that there are graphs of diameter 4 and radius 2 (respt. $\Delta (G)= \lambda (G)$), verifying $\lambda (\overline{G}) =\lambda (G)  + 1$. 
The graph displayed in  Figure \ref{hexagons} is an example of graph satisfying $rad(G)=2$, $diam(G)=4$ and $\lambda (\overline{G}) =\lambda (G)  + 1$, and 
 the complete graph $K_n$ is an example of a graph such that $\Delta (G)= \lambda (G)$ and $\lambda (\overline{G}) =\lambda (G)  + 1$, since   $\lambda (\overline{K_n}) =n$,  $\lambda (K_n)=\Delta (K_n)=n-1$.

%%%%%%%%%%%%%%%%%%%%%%%%%%%%%%%%%%%%%%%%%%%%%%%%%%%%%
%%%%%%%%%%%%%%%%%%%%%%%%%%%%%%%%%%%%%%%%%%%%%%%%%%%%%
%%%%%%%%%%%%%%%%%%%%%%%%%%%%%%%%%%%%%%%%%%%%%%%%%%%%%
%%%%%%%%%%%%%%%%%%%%%%%%%%%%%%%%%%%%%%%%%%%%%%%%%%%%%
%%%%%%%%%%%%%%%%%%%%%%%%%%%%%%%%%%%%%%%%%%%%%%%%%%%%%
%%%%%%%%%%%%%%%%%%%%%%%%%%%%%%%%%%%%%%%%%%%%%%%%%%%%%
\section{The global location-domination number}

\begin{defi}\rm

The \emph{global location-domination number} of a graph $G$, denoted by $\lambda_g (G)$, is defined as the minimum cardinality of a global LD-set of $G$. 
\end{defi}

Notice that, for every graph $G$, $\lambda_g (\overline{G})=\lambda_g (G)$, 
since for  every set of vertices $S\subset V(G)=V(\overline{G})$, $S$ is a global LD-set of  $G$ if and only if it  is a global LD-set of   $\overline{G}$.

%%%%%%%%%%%%%%%%%%%%%%%%%%%%%%%%%%%%%%%%%%%%%%%%%%%%%
\begin{proposition}\label{pro.ldglobal}
For any graph $G=(V,E)$,  $ \lambda (G) \le \lambda_g (G)  \le  \lambda (G)+1.$
\end{proposition}
\begin{proof}
The first inequality is a consequence of the fact that a global LD-set of $G$ is also an LD-set of $G$.
For the second inequality, suppose that $S$ is an LD-code of $G$, i.e. $|S|=\lambda (G)$. If $S$ is a global LD-set of $G$, then $\lambda_g(G)=\lambda(G)$. 
Otherwise, there exists a vertex $u\in V\setminus S$ dominating $S$ and $S\cup \{ u \}$
is an LD-set of $\overline{G}$. Therefore, $\lambda_g(G)\le \lambda (G) +1$. 
\end{proof}
%%%%%%%%%%%%%%%%%%%%%%%%%%%%%%%%%%%%%%%%%%%%%%%%%%%%%

%%%%%%%%%%%%%%%%%%%%%%%%%%%%%%%%%%%%%%%%%%%%%%%%%%%%%
\begin{corollary}
For any graph $G=(V,E)$,  $ \max \{ \lambda (G), \lambda (\overline{G}) \} \le \lambda_g (G)  \le \min \{ \lambda ({G})+1, \lambda (\overline{G})+1\}.$
\end{corollary}

%%%%%%%%%%%%%%%%%%%%%%%%%%%%%%%%%%%%%%%%%%%%%%%%%%%%%
\begin{corollary}\label{cor.lambdaglobalmayor}
Let $G=(V,E)$ be a graph.
\begin{itemize}

\item If $\lambda (G)\not= \lambda (\overline{G})$, then $\lambda_g(G)=\max \{ \lambda (G), \lambda (\overline{G})\}$.

\item If $\lambda (G)=\lambda (\overline{G})$, then $\lambda_g(G)\in \{ \lambda (G), \lambda (G) +1 \}$, and both possibilities are feasible.

\end{itemize}

\end{corollary}
\begin{proof}
Both statements are consequence of Proposition \ref{pro.ldglobal}.
Next, we give some examples to illustrate all possibilities given. It is easy to check that the complete graph $K_2$ satisfies
$1=\lambda (K_2)\not= \lambda (\overline{K_2})=2$ and $\lambda_g ({K_2})=\lambda (\overline{K_2})$; the path of order $P_3$,  satisfies $\lambda (P_3)=\lambda (\overline{P_3})=\lambda_g (P_3)=2$ and the cycle $C_5$, satisfies $\lambda (C_5)=\lambda (\overline{C_5})=2$ and $\lambda_g (C_5)=3$.
\end{proof}
%%%%%%%%%%%%%%%%%%%%%%%%%%%%%%%%%%%%%%%%%%%%%%%%%%%%%

\begin{proposition}\label{pro.lambdaglobal}
 For any graph $G=(V,E)$,  $ \lambda_g (G) = \lambda ({G})+1$ if and only if every LD-code of $G$ is non-global.
\end{proposition}
\begin{proof}
  A global LD-code of $G$ is an LD-set of both $G$ and $\overline{G}$. Hence, if $G$ contains at least a global LD-code, then $ \lambda_g (G) = \lambda ({G})$.
  Conversely, if every LD-code of $G$ is non-global, then there is no global LD-set of $G$ of size $\lambda (G)$. Then, $ \lambda_g (G) = \lambda ({G})+1$.
\end{proof}

In Tables \ref{tab.valorspetits} and  \ref{tab.valors},  the location-domination number of some families of graphs is displayed, along with the location-domination number of its complement graphs and the global location-domination number. 
Concretely, we consider the path $P_n$ of order  $n\ge1$; the cycle $C_n$ of order  $n\ge4$;  the wheel  $W_n$  of order  $n\ge5$, obtained by joining a new vertex to all vertices of a cycle of order $n-1$; the complete graph  $K_n$ of order  $n\ge3$; the complete bipartite graph $K_{r,s}$ of order  $n=r+s\ge4$, with $2\le r\le s$ and stable sets of order $r$ and $s$, respectively; the star $K_{1,n-1}$ of order  $n\ge4$, obtained by joining a new vertex to $n-1$ isolated vertices; and finally, the bi-star $K_2(r,s)$ of order $n=r+s+2\ge6$ with  $2\le r\le s$,  obtained by joining the central vertices of two stars $K_{1,r}$ and $K_{1,s}$ respectively.

As a consequence of Propositions \ref{andreu} and \ref{pro.lambdaglobal}, the following corollary holds.

\begin{corollary}
If $G$ is a graph with $diam(G)\ge 5$, then $\lambda_g(G)=\lambda(G)$.
\end{corollary}

We finalize this section by calculating  $\lambda(G)$,  $\lambda (\overline{G})$ and  $\lambda_g(G)$  for some basic graph families.

\begin{lemma}\label{cnpnpn-1}
If $n\ge 7$, then $\lambda (\overline{C_n})=\lambda (\overline{P_n})=\lambda (P_{n-1}).$
\end{lemma}
\begin{proof}
Firsty, we prove  that $\lambda (\overline{C_n})\le \lambda (P_{n-1})$ and $\lambda (\overline{P_n})\le \lambda (P_{n-1})$.
Suppose that  $V(P_{n-1})=\{ 1,2,\dots ,n-1\}$ and $E(P_{n-1})=\{ (i,i+1) : i=1,2,...,n-2\}$ are the vertex set and the edge set of $P_{n-1}$, respectively.
Assume that $S$ is an LD-code of $P_{n-1}$ such that $S$ does not contain vertex $1$ neither $n-1$ (it is easy to construct such an LD-code from those given in \cite{bchl}).  
Since $n-1\ge 6$, $S$ has at least $3$ vertices and there is no vertex in $V(P_{n-1})\setminus S$ dominating $S$ in $P_{n-1}$. 
Hence, $S$ is an LD-set of $\overline{P_{n-1}}$. 

Next,  consider the graph $G^*$ obtained by adding to the graph $\overline{P_{n-1}}$ a new vertex $u$ adjacent to the vertices $2,3,\dots ,n-2$, and may be to $1$ or $n-1$. 
Clearly, by construction,  $u$ is adjacent to all vertices of $S$ in $G^*$ and there is no vertex in $\overline{P_{n-1}}$ adjacent to all vertices in $S$.
Therefore, $S$ is an LD-set of $G^*$ and $\lambda (G^*)\le \lambda (P_{n-1})$.
Finally,  observe that if $u$ is not adjacent to $1$, neither to $n-1$, then $G^*$ is the graph $\overline{C_n}$ and if $u$ is  adjacent to exactly one of the vertices $1$ or $n-1$, then $G^*$ is the graph $\overline{P_n}$, which proves the inequalities before stated.

Lastly,  we  prove that $\lambda (P_{n-1}) \le \lambda (\overline{G})$, when $G\in \{  P_n , C_n \}$. 
Consider an LD-code $S$ of $\overline{G}$. 
Let $x$ be the only vertex dominating $S$ in $\overline{G}$, if it exists, or any vertex not in $S$, otherwise. 
By construction, $S$ is an LD-set of $G-x$, hence $\lambda (G-x)\le \lambda (\overline{G})$. To end the proof,  we distinguish two cases.

\begin{itemize}

\item[-] If $G$ is the cycle $C_n$, then $G-x$ is the path $P_{n-1}$, implying that $\lambda (P_{n-1})\le \lambda (\overline{C_n})$.

\item[-] If $G$ if the path $P_n$, then $G-x$ is either the path $P_{n-1}$ or the graph $P_r + P_s$, with $r,s\ge 1$ and $r+s=n-1\ge6$. Since, 
$ \lambda (P_r + P_s)= \lambda (P_r) + \lambda (P_s)=\lceil 2r/5 \rceil + \lceil 2s/5 \rceil \ge \lceil 2(r+s)/5 \rceil = \lambda (P_{n-1})$,
we conclude that, in any case,  $\lambda (P_{n-1})\le \lambda (\overline{P_n})$.
\end{itemize}
\vspace{-1cm}\end{proof}

\begin{proposition}\label{donosti}
 Let   $G$ be a graph of order $n\ge1$.
 If $G$ belongs to the set $\{P_n,C_n,W_n,K_n,K_{1,n-1},K_{r,n-r},K_2(r,n-r)\}$, 
then the values of $\lambda (G)$ and $\lambda (\overline{G})$ are known and they are displayed in Tables \ref{tab.valorspetits} and  \ref{tab.valors}.
\end{proposition}
\begin{proof}
  The values of the location-domination number of all these families, except the wheels,  are already known (see \cite{bchl,ours3,slater88}).
Next, let us calculate the values of the location-domination number for the wheels and for the complements of all these families and also, 
from the results previously proved, the global location-domination number of them. 

\begin{itemize}

\item For paths, cycles and wheels of small order, the values of $\lambda (G)$ and $\lambda_g (G)$ can easily be checked by hand (see Table \ref{tab.valorspetits}).

\item If $n\ge 7$, then $\lambda (W_n)=\lambda (C_{n-1})=\lceil \frac {2n-2}5 \rceil$, since (i) $W_n=K_1\vee C_{n-1}$, (ii) every LD-code $S$ of $C_{n-1}$ is an LD-set of $W_n$, 
and (iii) every LD-code of  $C_{n-1}$ is global.

 \item $\lambda (\overline{K_n})=\lambda (K_1+\dots +K_1)=\lambda (K_1)+\dots +\lambda (K_1) =n$.

 \item $\lambda (\overline{K_{1,n-1}})=\lambda (K_1+ K_{n-1})=\lambda (K_1)+\lambda (K_{n-1})=1+(n-2)=n-1$.

 \item $\lambda (\overline{K_{r,n-r}})=\lambda (K_r + K_{n-r})=\lambda (K_r) +\lambda(K_{n-r})=(r-1)+(n-r-1)=n-2, \,\textrm{ if }2\le r\le n-r$.

\item The complement of the bi-star $K_2(r,s)$, with $s=n-r$, is the graph obtained by joining a vertex $v$ to exactly $r$ vertices of a complete graph of order $r+s$ and joining a vertex $w$ to the remaining $s$ vertices of the complete graph of order $r+s$. It is immediate to verify that the set containing all vertices except $w$, a vertex adjacent to $v$ and a vertex adjacent to $w$ is an LD-code of $\overline{K_2(r,s)}$ with $n-3$ vertices. Thus, $\lambda (\overline{K_2(r,s)})=n-3$.

\item  For every $n\ge7$,  $\lambda (\overline{P_n})=\lambda (\overline{C_n})=\lceil \frac {2n-2}5 \rceil$. 
This result is a direct consequence of Lemma \ref{cnpnpn-1} and the fact that  $\lambda(P_n)=\lambda(C_n)=\lceil \frac {2n}5 \rceil$.

\item According to Lemma \ref{cnpnpn-1}, $\lambda (\overline{W_n})= \lambda (K_1 + \overline{C_{n-1}})=  \lambda (K_1) + \lambda(\overline{C_{n-1}})=1+\lambda (P_{n-2})=1+\lceil  2(n-2)/5 \rceil=\lceil  {(2n+1)}/5 \rceil.$

\end{itemize}
\vspace{-.6cm}\end{proof}

%%%%%%%%%%%%%%%%%%%%%%%%%%%%%%%%%%%%%%%%%%%%%%%%%%%%%%%
\begin{theorem}\label{rosellon}
Let  $G$ be a graph of order $n\ge1$. 
If $G$ belongs to the set $\{P_n,C_n,W_n,K_n,K_{1,n-1},K_{r,n-r},K_2(r,n-r)\}$, 
then $\lambda_g (G)$  is known and it is displayed in Tables \ref{tab.valorspetits} and  \ref{tab.valors}.
\end{theorem}
\begin{proof}
 By Corollary \ref{cor.lambdaglobalmayor},
$\lambda_g (K_n)=n$ and $\lambda_g (K_2(r,s))=n-2$. Since graphs $P_n$,  $C_n$, $W_n$, $K_{1,n-1}$,  $K_{r,s}$ and $K_2(r,n-r)$ contain at least an LD-global code,  by Proposition 
\ref{pro.lambdaglobal} we have $\lambda_g(G)=\max \{ \lambda (G), \lambda (\overline{G}) \}$ for all of them.
\end{proof}
%%%%%%%%%%%%%%%%%%%%%%%%%%%%%%%%%%%%%%%%%%%%%%%%%%%%%%%

\vspace{1cm}
\begin{table}
\begin{center}
  \begin{tabular}{c||cccccc|ccc|ccc}
          &  $P_1$  &  $P_2$  &  $P_3$  &  $P_4$   &  $P_5$  &  $P_6$  & $C_4$ & $C_5$ & $C_6$ & $W_5$ & $W_6$ & $W_7$
        \\
                  \hline
         $\lambda (G)$            & 1 & 1 & 2 & 2 & 2 & 3 & 2 & 2 & 3 & 2  & 3  & 3  \\
         $\lambda (\overline{G})$ & 1 & 2 & 2 & 2 & 2 & 3  & 2 & 2 & 3 & 3  & 3  &  4   \\
         $\lambda_g (G)=\lambda_g (\overline{G})$          & 1 & 2 & 2 & 2 & 3 & 3 &  2 & 3 & 3 & 3  &  3 &  4   \\
         \hline
\end{tabular}
\end{center}
\caption{The values of  $\lambda (G)$, $\lambda (\overline{G})$ and $\lambda_g (G)$ of small paths, cycles and wheels.}
\label{tab.valorspetits}
\end{table}

\begin{table}[hbt]
\begin{center}
  \begin{tabular}{c||ccc|c|cc|c}
          &  $P_n$, {\scriptsize $n\ge 7$} &  $C_n$, {\scriptsize $n\ge 7$} & $W_n$, {\scriptsize $n\ge 8$} & $K_n$, {\scriptsize $n\ge 2$} & $K_{1,n-1}$, {\scriptsize $n\ge 4$} &  $K_{r,n-r}$, {\scriptsize $2\le r\le n-r$}  & $K_2${\scriptsize$(r,n-r)$}, {\scriptsize $2\le r\le n-r$} \\
                  \hline
         $\lambda (G)$            & $\lceil \frac {2n}5 \rceil$   & $\lceil \frac {2n}5 \rceil$ & $\lceil \frac {2n-2}5 \rceil$  & $n-1$ &  $n-1$ & $n-2$ & $n-2$ \\ 
         $\lambda (\overline{G})$ & $\lceil \frac {2n-2}5 \rceil$ & $\lceil \frac {2n-2}5 \rceil$ & $ \lceil \frac {2n+1}5 \rceil$  & $n$ & $n-1$ & $n-2$ & $n-3$ \\
         $\lambda_g (G)=\lambda_g (\overline{G})$         & $\lceil \frac {2n}5 \rceil$  & $\lceil \frac {2n}5 \rceil$  &  $ \lceil \frac {2n+1}5 \rceil$ & $n$    & $n-1$ & $n-2$  & $n-2$ \\
         \hline
\end{tabular}
\end{center}
\caption{The values of  $\lambda (G)$, $\lambda (\overline{G})$ and $\lambda_g (G)$ for some families of graphs.}
\label{tab.valors}
\end{table}

\newpage
%%%%%%%%%%%%%%%%%%%%%%%%%%%%%%%%%%%%%%%%%%%%%%%%%%%%%%%
%%%%%%%%%%%%%%%%%%%%%%%%%%%%%%%%%%%%%%%%%%%%%%%%%%%%%%%
%%%%%%%%%%%%%%%%%%%%%%%%%%%%%%%%%%%%%%%%%%%%%%%%%%%%%%%
%%%%%%%%%%%%%%%%%%%%%%%%%%%%%%%%%%%%%%%%%%%%%%%%%%%%%%%
%%%%%%%%%%%%%%%%%%%%%%%%%%%%%%%%%%%%%%%%%%%%%%%%%%%%%%%
\section{Global location-domination in block-cactus}

This section is devoted to characterizing those block-cactus $G$  satisfying $\lambda (\overline{G}) =\lambda (G)  + 1$.
By Proposition \ref{pro.globalImplica}, this equality is feasible only for graphs without global LD-codes.

We will refer in this section to some specific  graphs, such as  the \emph{paw}, the \emph{bull}; the \emph{banner}  $P$, 
the \emph{complement of the banner}, $\overline{P}$; the \emph{butterfly} and the \emph{corner} $\textsf{L}$ (see Figure \ref{fig.somegrafs}).

%%%%%%%%%%%%%%%%%
%%%%%%%%%%%%%%%%%
\begin{figure}[hbt]
\begin{center}
\includegraphics[height=1.7cm]{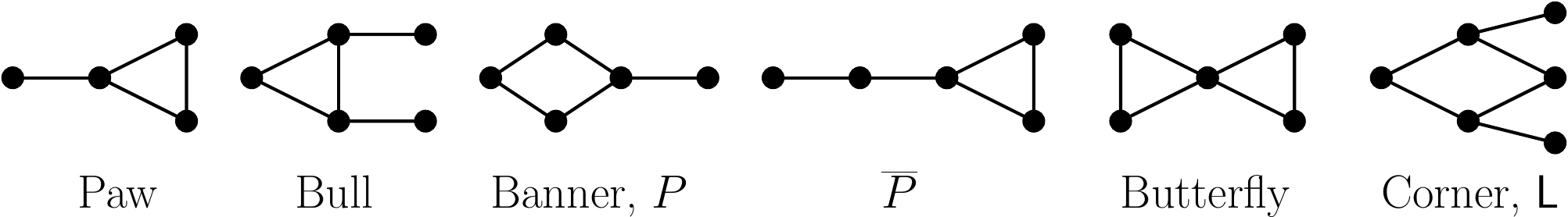}
\caption{Some special graphs.}\label{fig.somegrafs}
\end{center}
\end{figure}
%%%%%%%%%%%%%%%%%
%%%%%%%%%%%%%%%%%

The block-cactus of order at most 2 are $K_1$ and $K_2$.
For these graphs we have $\lambda (K_1)=\lambda (\overline{K_1})=1$ and $\lambda (K_2) =1 < 2=\lambda (\overline{K_2})$.

In \cite{cahemopepu12}, all  16 non-isomorphic graphs with $\lambda (G)=2$ are given.
After carefully examining all cases, the following result is obtained (see Figure \ref{fig.lambda2}).

%%%%%%%%%%%%%%%%%
%%%%%%%%%%%%%%%%%
\begin{figure}[hbt]
\begin{center}
\includegraphics[height=5cm]{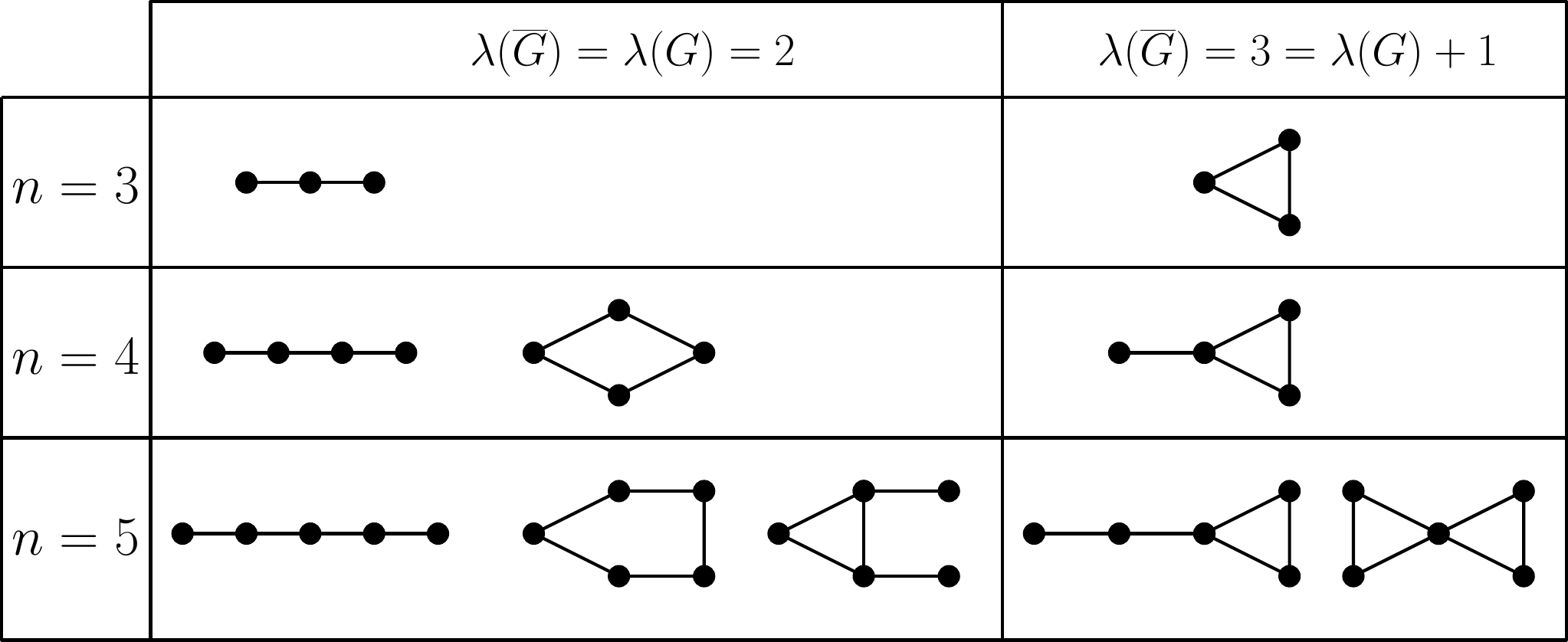}
\caption{All block-cactus with $\lambda (G)=2$.}
\label{fig.lambda2}
\end{center}
\end{figure}
%%%%%%%%%%%%%%%%%
%%%%%%%%%%%%%%%%%

%%%%%%%%%%%%%%%%%%%%%%%%%%%%%%%%%%%%%%%%%%%%%%%%%%%%%
\begin{proposition}\label{pro.lambda2}
  Let $G=(V,E)$ be a block-cactus such that $\lambda (G)=2$.
  Then,  $\lambda (\overline{G}) \ge\lambda (G)$. Moreover, $\lambda (\overline{G}) =\lambda (G)  + 1=3$ if and only if $G$ is isomorphic to the cycle of order 3, the paw, the butterfly or the complement of a banner.
\end{proposition}
%%%%%%%%%%%%%%%%%%%%%%%%%%%%%%%%%%%%%%%%%%%%%%%%%%%%%

Next, we approach the case $\lambda (G)\ge3$.  
First of all, let us present some lemmas, providing a number of necessary conditions for a given block-cactus  to have at least a non-global LD-set.

%%%%%%%%%%%%%%%%%%%%%%%%%%%%%%%%%%%%%%%%%%%%%%%%%%%%%%%
\begin{lemma}\label{lema1}%%%%%%%
Let $G=(V,E)$ be a block-cactus and $S\subseteq V$ a non-global LD-set of $G$.
If $u\in V\setminus S$ dominates $S$, then $G[N(u)]$ is a disjoint union of cliques.
\end{lemma}
\begin{proof}
Let $x,y$ be a pair of vertices belonging to the same component $H$ of $G[N(u)]$.
Suppose that $xy\not\in E$ and take an $x-y$ path $P$ in $H$.
Let $z$ be an inner vertex of $P$.
Notice that the set $\{u,x,y,z\}$ is contained in the same block $B$ of $G$.
As $B$ is not a clique, it must be a cycle, a contradiction, since $deg_B(u)\ge3$.
\end{proof}
%%%%%%%%%%%%%%%%%%%%%%%%%%%%%%%%%%%%%%%%%%%%%%%%%%%%%%%

%%%%%%%%%%%%%%%%%%%%%%%%%%%%%%%%%%%%%%%%%%%%%%%%%%%%%%%
\begin{lemma}\label{lema2}
Let $G=(V,E)$ be a block-cactus and $S\subseteq V$ a non-global LD-set of $G$.
If $u\in V\setminus S$ dominates $S$ and $W=V\setminus N[u]$, then, for every vertex $w\in W$, the  following properties hold.

 \begin{itemize}

   \item[i)]  $1\leq \mid N(u)\cap N(w)\mid \leq 2$.
	
	\item[ii)]  If $N(u)\cap N(w)=\{x\}$, then $x\in S$.
	
   \item[iii)]  If $N(u)\cap N(w)=\{x,y\}$, then $xy\not\in E$.

   \item[iv)] If $w'\in W$ and $N(u)\cap N(w)=N(u)\cap N(w')=\{x\}$, then $w'=w$.
	
   \item[v)] If $w'\in W$ and $|N(u)\cap N(w)|=|N(u)\cap N(w')|=2$, then $N[w]\cap N[w']=\emptyset$.
	
 \end{itemize}
 \end{lemma}
 \begin{proof}
i),ii),iii): $\mid N(u)\cap N(w)\mid \geq 1$ as $S\subset N(u)$ and $S$ dominates vertex $w$.
	If  $N(u)\cap N(w)=\{x\}$, then necessarily $x\in S$.
	Assume that $ \mid N(u)\cap N(w)\mid >1$.
	Observe that the set $N[u]\cap N[w]$ is contained in the same block $B$ of $G$.
	Certainly, $B$ must be a cycle since $uw\not\in E$.
	Hence, $\mid N(u)\cap N(w)\mid =2$.
	Moreover, in this case $B$ is isomorphic to the cycle $C_4$, which means that, if $V(B)=\{u,x,y,w\}$,  then $xy\not\in E$.
	
\noindent iv): If  $w'\neq w$, then $S\cap N(w)\neq S\cap N(w')$, as $S$ is an LD-set.

\noindent  v): Suppose that $w\neq w'$, $ N(u)\cap N(w)=\{x,y\}$ and  $N(u)\cap N(w')=\{z,t\}$.
Notice that $\{x,y\}\neq \{z,t\}$, since $S$ is an LD-set.
If $y=z$, then the set $\{u,w,w',x,y,t\}$ is contained in the same block $B$ of $G$, a contradiction,
because  $B$  is neither a clique, since $uw\not\in E$, nor a cycle, as $deg_G(u)\ge3$.
Assume thus that $\{x,y\}\cap \{z,t\}=\emptyset$.
If either $ww'\in E$ or $N(w)\cap N(w')\neq\emptyset$, then the set $\{u,w,w',x,y,z, t\}$ is contained in the same block $B$ of $G$, again a  contradiction,
because  $B$ is neither a clique, since $uw\not\in E$, nor a cycle, as $deg_G(u)\ge4$.
\end{proof}
%%%%%%%%%%%%%%%%%%%%%%%%%%%%%%%%%%%%%%%%%%%%%%%%%%%%%%%

%%%%%%%%%%%%%%%%%%%%%%%%%%%%%%%%%%%%%%%%%%%%%%%%%%%%%%%
 \begin{lemma}
Let $G=(V,E)$ be a block-cactus and $S\subseteq V$ a non-global LD-set of $G$.
If $u\in V\setminus S$ dominates $S$ and $W=V\setminus N[u]$, then
 \begin{itemize}
   \item Every component of $G[W]$ is isomorphic either to $K_1$ or to $K_2$.
 \item If $w,w'\in W$ and $ww'\in E$, then the set $\{w,w'\}$ is contained in  the same block,  which is isomorphic to $C_5$.
 \end{itemize}
 \end{lemma}
 \begin{proof}
Let $w,w'$ such that $ww'\in E$.
According to Lemma \ref{lema2}, the set  $\{u\}\cup N[w]\cup N[w']$ forms a block $B$ of $G$, which is isomorphic to the cycle $C_5$.
In particular, no vertex of $W\setminus\{w,w'\}$ is adjacent neither to $w$ nor to $w'$.
 \end{proof}
%%%%%%%%%%%%%%%%%%%%%%%%%%%%%%%%%%%%%%%%%%%%%%%%%%%%%%%

As a corollary of the previous three lemmas the following proposition is obtained.

%%%%%%%%%%%%%%%%%%%%%%%%%%%%%%%%%%%%%%%%%%%%%%%%%%%%%
\begin{proposition}\label{propblocks}
  Let $G=(V,E)$ be a block-cactus and $S\subseteq V$ a non-global LD-set of $G$.
	
If $u\in V\setminus S$ dominates $S$, then $G$ can be obtained by identifying the vertex $u$ of some copies of each of the following graphs  (see Figure \ref{fig.blocs}):
   \begin{itemize}
     \item[a)] $u$ is adjacent to every vertex of a complete graph $K_r$,  $r\ge 1$,  and each one of the vertices of $K_r$ is adjacent to at most a new vertex of degree $1$;
     \item[b)] $u$ is a vertex of a cycle of order $4$, and each neighbor of $u$ is adjacent to at most a new vertex of degree $1$;
     \item[c)] $u$ is a vertex of a cycle of order 5.
   \end{itemize}
   \end{proposition}
%%%%%%%%%%%%%%%%%%%%%%%%%%%%%%%%%%%%%%%%%%%%%%%%%%%%%

%%%%%%%%%%%%%%%%%
%%%%%%%%%%%%%%%%%
\begin{figure}[htb]
\begin{center}
\includegraphics[height=3.5cm]{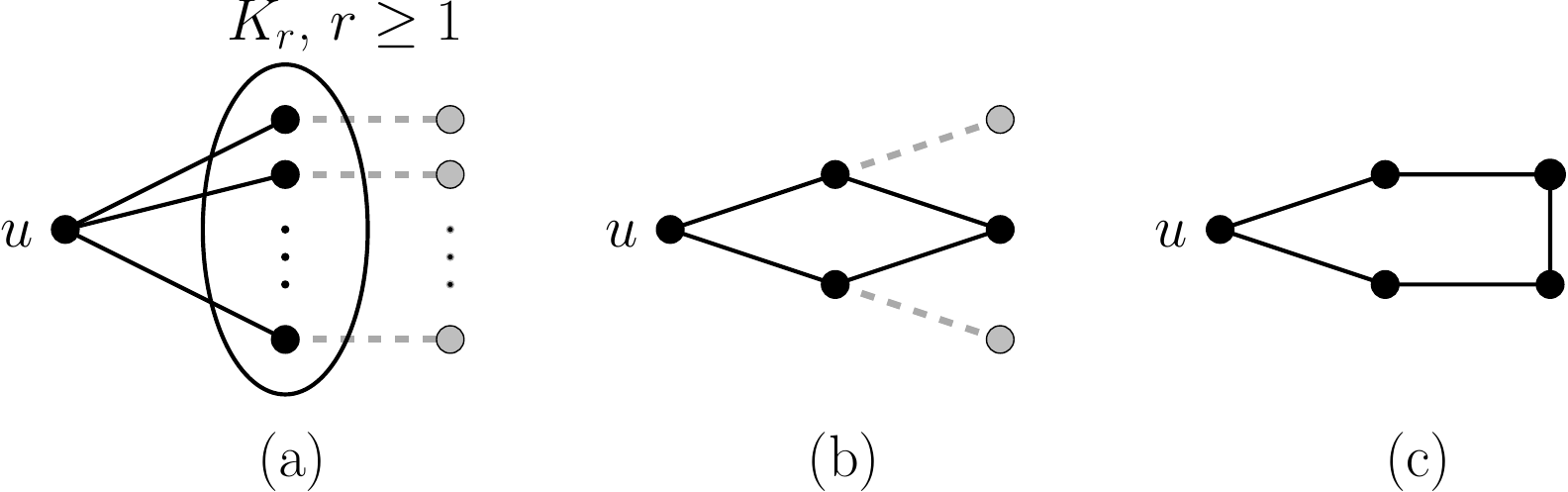}
\caption{Induced subgraphs of a block-cactus having a non-global LD-set whose dominating vertex is $u$. Gray vertices are optional.}\label{fig.blocs}
\end{center}
\end{figure}
%%%%%%%%%%%%%%%%%
%%%%%%%%%%%%%%%%%

In the next theorem, we characterize those block-cactus not containing any global LD-code of order at least $3$.

%%%%%%%%%%%%%%%%%%%%%%%%%%%%%%%%%%%%%%%%%%%%%%%%%%%%%%%
\begin{theorem}\label{pro.noglobal}
Let $G=(V,E)$ be a block-cactus such that $\lambda (G)\ge 3$.
Then, every LD-code of $G$ is non-global if and only if $G$ is isomorphic to one of the following graphs (see Figure \ref{fig.NoGlobalLDset}):
  \begin{itemize}
    \item[a)] $K_1 \vee (K_1 + K_r)$, $r\ge 3$;

    \item[b)] the graph obtained by joining one vertex of $K_2$ with a vertex of a complete graph of order $r+1$, $r\ge 3$;

    \item[c)] $K_{r+1}$, $r\ge 3$;

    \item[d)] the graph obtained by joining a vertex of $K_2$ with one of the vertices of degree 2 of a corner;

    \item[e)] if we consider the graph  $K_1\vee (K_{r_1}+ \dots+ K_{r_{t}})$ and $t'$ copies of a corner, with $t+t'\ge 2$ and $r_1,\dots , r_{t} \ge 2$, the graph obtained by identifying the vertex $u$ of $K_1$ with one of the vertices of degree 2 of each copy of the corner.
  \end{itemize}
\end{theorem}
%%%%%%%%%%%%%%%%%%%%%%%%%%%%%%%%%%%%%%%%%%%%%%%%%%%%%%%

%%%%%%%%%%%%%%%%%
%%%%%%%%%%%%%%%%%
\begin{figure}[htb]
\begin{center}
\includegraphics[height=5.5cm]{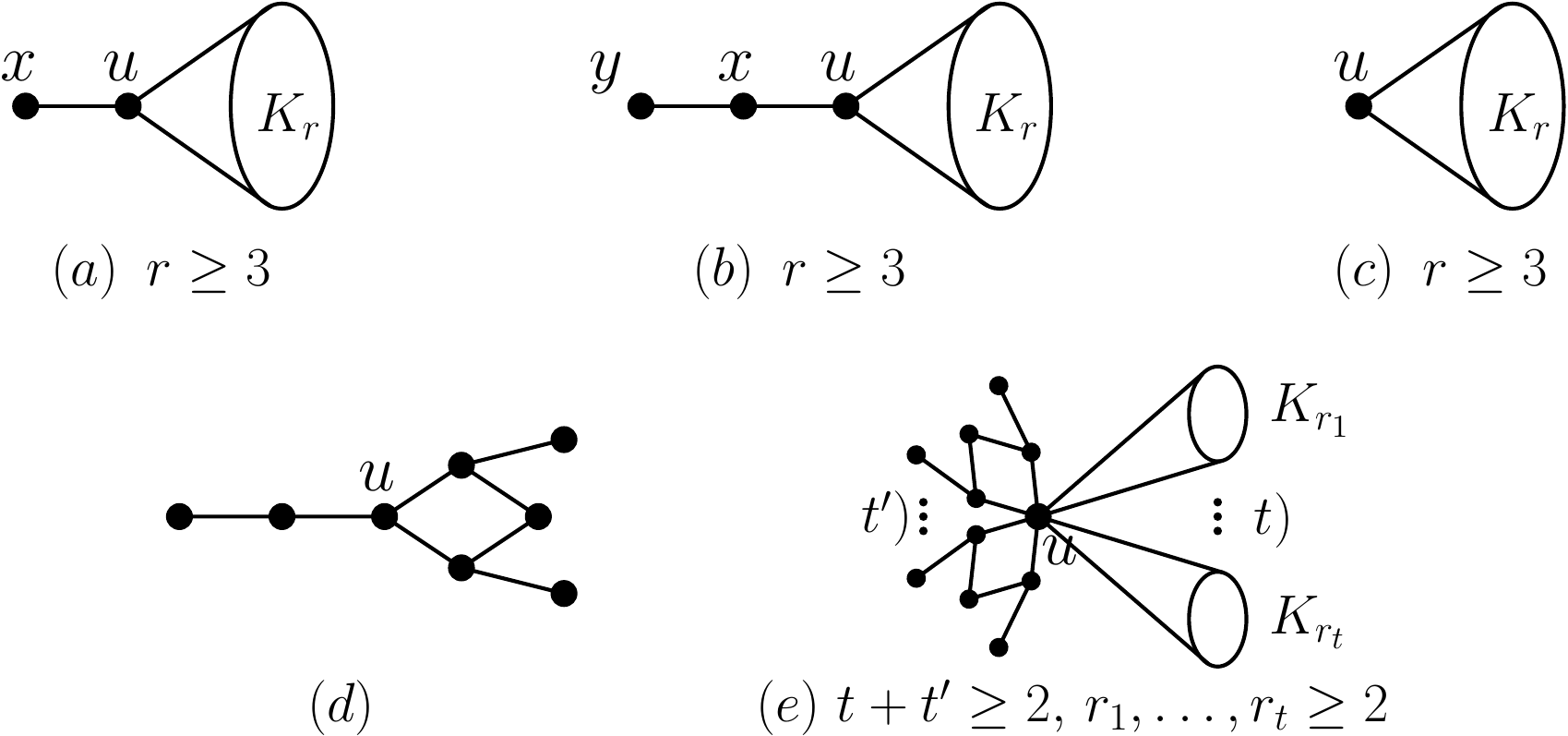}
\caption{Block-cactus with $\lambda(G)\ge3$ not containing any global LD-code.}
\label{fig.NoGlobalLDset}
\end{center}
\end{figure}
%%%%%%%%%%%%%%%%%
%%%%%%%%%%%%%%%%%

%%%%%%%%%%%%%%%%%%%%%%%%%%
\begin{proof} Firstly, let us show that none of these graphs contains a global LD-code.
\begin{itemize}

    \item[a)] Let $G$ be the graph showed in Figure  \ref{fig.NoGlobalLDset}(a).
    Observe that $\lambda(G)=r$ and, for every LD-code $S$, $|S\cap \{x,u\}|=1$ and $|S\cap K_r|=r-1$.
	  Let $w$ be the vertex of $K_r$ not in $S$.
		If $x\in S$, then $S\subset N(u)$. Otherwise, if  $u\in S$, then $S\subset N(w)$.

    \item[b)]  Let $G$ be the graph showed in Figure  \ref{fig.NoGlobalLDset}(b).
		Notice that  $\lambda(G)=r$ and, for every LD-code $S$, $x\in S$ and $|S\cap K_r|=r-1$.
	 Hence , if $S$ is an  LD-code of $G$, then  $S\subset N(u)$.

    \item[c)]  If $G=K_n$ (Figure \ref{fig.NoGlobalLDset}(c)), then $G$ contains no global LD-code.

    \item[d)]  Let $G$ be the graph showed in Figure  \ref{fig.NoGlobalLDset}(d).
    Clearly, the unique LD-code of  $G$ is $S=N(u)$.

   \item[e)]   Let $G$ be the graph showed in Figure   \ref{fig.NoGlobalLDset}(e).
	  In this graph, every LD-code contains both  vertices  adjacent to vertex $u$ in each copy of the corner and, for every $i\in\{1,\dots, t\}$,  $r_i -1$ vertices of $K_{r_i}$.
		Thus,  for every  LD-code $S$ of $G$,  $S\subset N(u)$.
	
\end{itemize}

In order to prove that these are the only graphs not containing any global LD-code, we previously need to show the following lemmas.

%%%%%%%%%%%%%%%%%%%%%%%%%%%%%%%%%%%%%%%%%%%%%%%%%%%%%%%
\begin{lemma} \label{l0}
Let $G=(V,E)$ be a block-cactus and $S\subseteq V$ a non-global LD-set of $G$.
If $u\in V\setminus S$ dominates $S$, then, for every component $H$ of $G[N(u)]$ of cardinality $r$, $|V(H)\cap S)|=\max\{1,r-1\}$.
\end{lemma}

\begin{proof}
This result is an immediate consequence of Lemma \ref{lema1} ( $G[N(u)]$ is a disjoint union of cliques), along with the fact that $S$ is an LD-set.
\end{proof}
%%%%%%%%%%%%%%%%%%%%%%%%%%%%%%%%%%%%%%%%%%%%%%%%%%%%%%%

Given a cut vertex $u$ of a connected graph $G$, let  $\Lambda_u$ be the set of all maximal connected subgraphs $H$ of $G$ such that 
(i)  $u\in V(H)$ and (ii) $u$ is not a cut vertex of $H$.
Observe that any subgraph of  $\Lambda_u$ can be obtained from a certain component of the graph $G-u$, by adding the vertex $u$ according to the structure of $G$.

%%%%%%%%%%%%%%%%%%%%%%%%%%%%%%%%%%%%%%%%%%%%%%%%%%%%%%%
\begin{lemma}\label{l1}
Let $G=(V,E)$ be a block-cactus with $\lambda(G)\ge3$ and let $S\subseteq V$ be a non-global LD-set of $G$.
If $u\in V\setminus S$ dominates $S$ and the set  $\Lambda_u$ contains a graph  isomorphic to one of the graphs displayed in Figure \ref{fig.nueva}, then  $G$ has  a global  LD-code.
\end{lemma}
%%%%%%%%%%%%%%%%%%%%%%%%%%%%%%%%%%%%%%%%%%%%%%%%%%%%%%%

%%%%%%%%%%%%%%%%% figura versió merce
%%%%%%%%%%%%%%%%%
\begin{figure}[hb]
\begin{center}
\includegraphics[height=2.7cm]{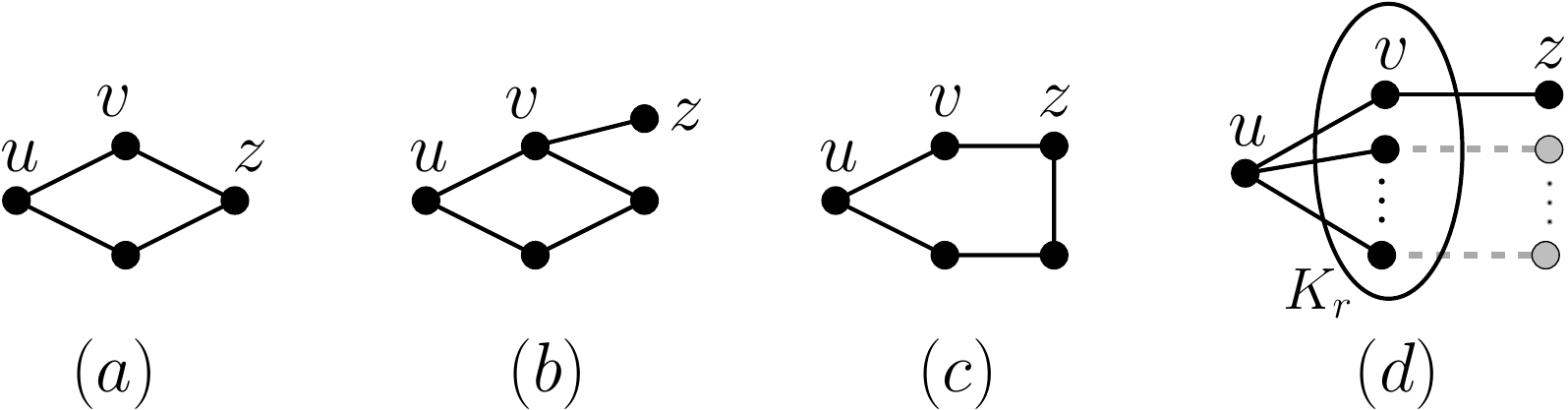}
%\caption{Obtaining a global LD-code of $G$.}\label{fig.blocsGlobal}
\caption{Some possible elements of $\Lambda_u$.}
\label{fig.nueva}
\end{center}
\end{figure}
%%%%%%%%%%%%%%%%%
%%%%%%%%%%%%%%%%%

%%%%%%%%%%%%%%%%%%%%%%%%%%%%%%%%%%%%%%%%%%%%%%%%%%%%%%%
\begin{proof}
Let $v,z$ the pair of vertices shown in  Figure \ref{fig.nueva}.
Then, according to Lemma \ref{l0}, $v\in S$ and  $S'=(S\setminus\{v\}) \cup \{z\}$ is an  LD-set de $G$ having the same cardinality as $S$. Hence, $S'$ is a global  LD-code of $G$.
\end{proof}
%%%%%%%%%%%%%%%%%%%%%%%%%%%%%%%%%%%%%%%%%%%%%%%%%%%%%%%

%%%%%%%%%%%%%%%%%%%%%%%%%%%%%%%%%%%%%%%%%%%%%%%%%%%%%%%
\begin{lemma}\label{l2}
Let $G=(V,E)$ be a block-cactus with $\lambda(G)\ge3$ and let $S\subseteq V$ be a non-global LD-set of $G$.
If $u\in V\setminus S$ dominates $S$ and the set $\Lambda_u$ contains a pair of graphs $H_1$ and $H_2$ such that 
$H_1,H_2\in\{P_2,P_3\}$, then   $G$ has  a global  LD-code.
\end{lemma}
\begin{proof}
If $H_1$ is isomorphic to $P_3$,  with $V(H_1)=\{u,v,z\} $ and $E(H_1)=\{uv,vz\}$, then, according to Lemma \ref{l0}, $v\in S$ and  $S'=(S\setminus\{v\}) \cup \{z\}$ is an  LD-set de $G$ having the same cardinality as $S$.
Hence, $S'$ is a global  LD-code of $G$.

If both $H_1$ and $H_2$ are isomorphic to $P_2$,  and  $V(H_1)=\{u,t\} $ and $E(H_1)=\{ut\}$, then, according to Lemma \ref{l0}, $v\in S$ and  $S'=(S\setminus\{t\}) \cup \{u\}$ is an  LD-set de $G$ having the same cardinality as $S$.
Hence, $S'$ is a global  LD-code of $G$.
\end{proof}
%%%%%%%%%%%%%%%%%%%%%%%%%%%%%%%%%%%%%%%%%%%%%%%%%%%%%%%

%%%%%%%%%%%%%%%%%%%%%%%%%%%%%%%%%%%%%%%%%%%%%%%%%%%%%%%
\begin{lemma}\label{l3}
Let $G=(V,E)$ be a block-cactus and $S\subseteq V$ a non-global LD-set of $G$ whose dominating vertex is $u$.
If $\Lambda_u$ contains three graphs $H_1$, $H_2$ and $H_3$ such that  $H_1\in\{P_2,P_3\}$ and $H_2,H_3\in\{K_r,\textsf{L}\}$, where $\textsf{L}$ denotes the corner graph displayed in Figure \ref{fig.somegrafs}, then  $G$ has  a global  LD-code.
\end{lemma}
\begin{proof}
If $H_1$ is isomorphic to $P_2$,  with $V(H_1)=\{u,t\} $ and $E(H_1)=\{ut\}$, then, according to Lemma \ref{l0}, $v\in S$ and  $S'=(S\setminus\{t\}) \cup \{u\}$ is an  LD-set de $G$ having the same cardinality as $S$.
Hence, $S'$ is a global  LD-code of $G$.

If $H_1$ is isomorphic to $P_2$,  $V(H_1)=\{u,v,z\} $ and $E(H_1)=\{uv,vz\}$, then,  according to Lemma \ref{l0}, $v\in S$ and  $S'=(S\setminus\{v\}) \cup \{z\}$ is an  LD-set de $G$ having the same cardinality as $S$.
Hence, $S'$ is a global  LD-code of $G$.
\end{proof}
%%%%%%%%%%%%%%%%%%%%%%%%%%%%%%%%%%%%%%%%%%%%%%%%%%%%%%%

%%%%%%%%%%%%%%%%%%%%%%%%%%%%%%%%%%%%%%%%%%%%%%%
We are now ready to end the proof of the Theorem \ref{pro.noglobal}.
Suppose that $G$ is a block-cactus such that every LD-code of $G$ is non-global.
Let $S\subseteq V$ be an LD-code of $G$ and let $u\in V\setminus S$ be a vertex dominating  $S$.
Notice that, according to Proposition \ref{propblocks}, every graph of $\Lambda_u$ is isomorphic to one of the graphs displayed in Figure \ref{fig.blocs}.
Moreover, having into account the results obtained in Lemma \ref{l1}, Lemma \ref{l2} and  Lemma \ref{l3}, the set $\Lambda_u$ is one the following sets:

 \begin{itemize}

 \item $\{P_2,K_r\}$. In this case,  $G$ is the graph shown in Figure \ref{fig.NoGlobalLDset}(a).
	
 \item $\{P_3,K_r\}$. In this case,  $G$ is the graph shown in Figure \ref{fig.NoGlobalLDset}(b).
	
 \item $\{P_2,\textsf{L}\}$.  Let  $u,t$ be the vertices of $P_2$. Then, according to Lemma \ref{l0},  $t\in S$, and  $S'=(S\setminus\{t\}) \cup \{u\}$ is a global  LD-code of $G$.

 \item $\{P_3,\textsf{L}\}$. In this case,  $G$ is the graph shown in Figure \ref{fig.NoGlobalLDset}(d).

 \item $\{K_r\}$. In this case,  $G$ is the graph shown in Figure \ref{fig.NoGlobalLDset}(c).

 \item A set of cardinality at least two, being every graph isomorphic either to a clique or to a corner. In this case,  $G$ is a graph as  shown in Figure \ref{fig.NoGlobalLDset}(e).

 \end{itemize}
This completes the proof of Theorem \ref{pro.noglobal}.
\end{proof}
%%%%%%%%%%%%%%%%%%%%%%%%%%%%%%%%%%%%%%%%%%%%%%%

As an immediate consequence of  Propositions \ref{pro.lambdaglobal} and \ref{pro.lambda2} and Theorem \ref{pro.noglobal}, the following corollaries are obtained.

\begin{corollary}
A block-cactus $G$ satisfies $\lambda_g(G) = \lambda (G) +1$ if and only if $G$ is isomorphic either to one of the graphs described in Figure \ref{fig.NoGlobalLDset} 
or it belongs to the set $\{P_2,P_5,C_3,C_5,\overline{P},{\rm paw,bull, butterfly}\}$.
\end{corollary}

%%%%%%%%%%%%%%%%%%%%%%%%%%%%%%%%%%%%%%%%%%%%%%%%%%%%%%%
\begin{corollary}
  Every tree $T$ other than $P_2$ and $P_5$ satisfies $\lambda (T)=\lambda_g (T)$.
\end{corollary}
%%%%%%%%%%%%%%%%%%%%%%%%%%%%%%%%%%%%%%%%%%%%%%%%%%%%%%%

%%%%%%%%%%%%%%%%%%%%%%%%%%%%%%%%%%%%%%%%%%%%%%%%%%%%%%%
\begin{corollary}
  Every unicyclic graph  $G$ different from the one displayed in Figure \ref{fig.NoGlobalLDset}(d) and not belonging  
	to the set $\{C_3,C_5,\overline{P},{\rm paw,bull}\}$ satisfies $\lambda (G)=\lambda_g (G)$.
\end{corollary}
%%%%%%%%%%%%%%%%%%%%%%%%%%%%%%%%%%%%%%%%%%%%%%%%%%%%%%%

If $G$ is a block-cactus of order at least 2, we have obtained the following characterization.

%%%%%%%%%%%%%%%%%%%%%%%%%%%%%%%%%%%%%%%%%%%%%%%%%%%%%%%
\begin{theorem}
   If $G=(V,E)$ is a block-cactus of order at least 2, then
   $\lambda (\overline{G}) =\lambda (G)  + 1$ if and only if $G$ is isomorphic to one of the following graphs (see Figure \ref{fig.block-cactusGuanyaComplementari}):
   \begin{itemize}
    \item[(a)] $K_1 \vee (K_1 + K_r)$, $r\ge 2$;
    \item[(b)] the graph obtained by joining one vertex of $K_2$ with a vertex of a complete graph of order $r+1$, $r\ge 2$;
    \item[(c)] $K_{r+1}$, $r\ge 1$;
    \item[(d)] $K_1 \vee (K_{r_1} + \dots + K_{r_t})$, $t\ge 2$, $r_1,\dots ,r_t\ge 2$.
    \end{itemize}
\end{theorem}
%%%%%%%%%%%%%%%%%%%%%%%%%%%%%%%%%%%%%%%%%%%%%%%%%%%%%%%

%%%%%%%%%%%%%%%%%
%%%%%%%%%%%%%%%%%
\begin{figure}[htb]
\begin{center}
\includegraphics[height=3.0cm]{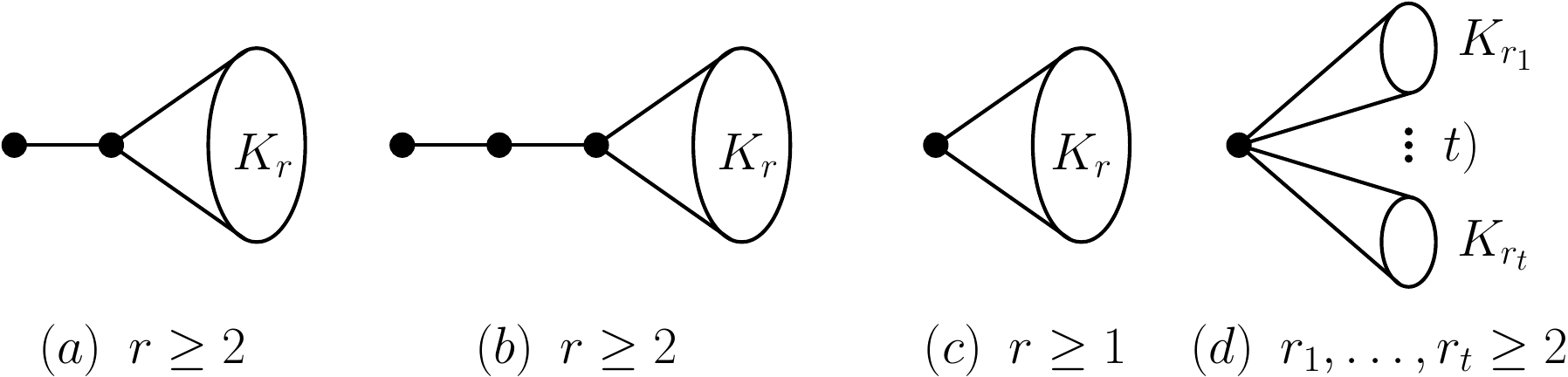}
\caption{Block-cactus satisfying $\lambda (\overline{G}) =\lambda (G)  + 1$.}
\label{fig.block-cactusGuanyaComplementari}
\end{center}
\end{figure}
%%%%%%%%%%%%%%%%%
%%%%%%%%%%%%%%%%%

%%%%%%%%%%%%%%%%%%%%%%%%%%%%%%%%%%%%%%%%%%%%%%%%%%%%%%%
\begin{proof}
  Let us see first that all graphs described above satisfy $\lambda (G)<\lambda (\overline{G})$. 
	Recall that if $W$ is a set of twin vertices of a graph $G$, then every LD-set must contain at least all but one of the vertices of $W$.
  Consider one of the  graphs described in (a), $G\cong K_1 \vee (K_1 + K_r)$, $r\ge 2$. 
	The complement of $G$ is the graph  $K_1+ K_{1,r}$. 
	It is easy to verify that $\lambda (G)=r<r+1=\lambda (\overline{G})$.
  If $G$ is one of the graphs described in b),  then $\lambda (G)=r<r+1=\lambda (\overline{G})$.
  Finally, if $G\cong K_1 \vee+ (K_{r_1} + \dots + K_{r_t})$ is a graph of order $n$, with $t\ge 1$ and $r_1,\dots ,r_t\ge 2$, then we have $\lambda (G)=n-t-1<n-t=\lambda (\overline{G})$.

Now, suppose that $G=(V,E)$ is a block-cactus of order at least 3 satisfying $\lambda (\overline{G}) =\lambda (G)  + 1$.

If $\lambda(G)=1$, as the order of G is at least 2, then $G$ is the 2-path $P_2$, which satisfies $2=\lambda(\overline P_2)=\lambda(P_2)+1$.
This case is described under (c)  when r=1 (see Figure8).

If $\lambda (G)=2$, then by Proposition \ref{pro.lambda2} the graph $G$ is the paw, the complement of the banner, the 3-cycle $C_3$ or the butterfly, and these graphs are described, respectively, 
under (a) when $r=2$; (b) when $r=2$; (c) when $t=1$ and $r_1=2$ and (d) when $t=r_1=r_2=2$ (see Figure \ref{fig.block-cactusGuanyaComplementari}).

If $\lambda (G)\ge3$, by Proposition \ref{pro.globalImplica}, $G$ does not contain a global LD-code, and therefore it must be one of those graphs described in Theorem \ref{pro.noglobal}. 
Hence, it suffices to prove that the graphs described under items d) or e) with $t'>0$, in Theorem \ref{pro.noglobal}, do not satisfy $\lambda (\overline{G}) =\lambda (G)  + 1$.
The graph $G$ described in item d) satisfies $\lambda (G)=\lambda (\overline{G})=3$, since an LD-code of $G$ is the set containing the three vertices adjacent to the three vertices of degree 1 in $G$ and an LD-code of $\overline{G}$ is the set containing the three vertices adjacent to the three vertices of degree 3 in $G$.
Finally, if $G$ is one of the graphs described in item e) obtained from $t$ copies of complete graphs and $t'$ copies of corners, $t'\ge 1$, then the set of vertices including all but one vertex of each complete graph and the two vertices of degree 3 of each copy of the corner, is an LD-code of $G$. 
If we change exactly one of the vertices of degree 3 of a copy of the corner by the vertex of degree 2 in this copy, then we obtain an LD-code of $\overline{G}$. 
Therefore, $\lambda (G)=\lambda (\overline{G})= 2t'+ (r_1-1)+\dots +(r_t-1)$.
\end{proof}
%%%%%%%%%%%%%%%%%%%%%%%%%%%%%%%%%%%%%%%%%%%%%%%%%%%%%%%

%%%%%%%%%%%%%%%%%%%%%%%%%%%%%%%%%%%%%%%%%%%%%%%%%%%%%%%
\begin{corollary}
  Every tree $T$ other than $P_2$  satisfies $\lambda(\overline{T})\le \lambda (T)$.
\end{corollary}
%%%%%%%%%%%%%%%%%%%%%%%%%%%%%%%%%%%%%%%%%%%%%%%%%%%%%%%

%%%%%%%%%%%%%%%%%%%%%%%%%%%%%%%%%%%%%%%%%%%%%%%%%%%%%%%
\begin{corollary}
  Every unicyclic graph  $G$ not beloging  to the set $\{C_3,\overline{P},{\rm paw}\}$ satisfies $\lambda(\overline{G})\le \lambda (G)$.
\end{corollary}
%%%%%%%%%%%%%%%%%%%%%%%%%%%%%%%%%%%%%%%%%%%%%%%%%%%%%%%

%\newpage
%%%%%%%%%%%%%%%%%%%%%%%%%%%%%%%%%%%%%%%%%%%%%%%%%%%%%%%
%%%%%%%%%%%%%%%%%%%%%%%%%%%%%%%%%%%%%%%%%%%%%%%%%%%%%%%
\section{Further research}
This work can be continued in several directions. Next, we propose a few of them.

\begin{itemize}

\item In this work, we have completely solved the equality $\lambda (\overline{G}) =\lambda (G)  + 1$ for the block-cactus family. 
In \cite{oursbip}, a similar study has been done for the family of bipartite graphs. We suggest to approach this problem for other families of graphs, such as outerplanar graphs, chordal graphs  and cographs.

\item Characterizing those trees $T$  satisfying $\lambda (\overline{T})= \lambda (T)=\lambda_g (T)$.

\item We have proved that every tree other than $P_2$ and $P_5$, every cycle other than $C_3$ and $C_5$, and every complete bipartite graph satisfies the  equality $\lambda (G)=\lambda_g (G)$. 
We propose  to find other families of graphs with this same  behaviour.

\end{itemize}

%%%%%%%%%%%%%%%%%%%%%%%%%%%%%%%%%%%%%%%%%%%%%%%%%%%%%%%
%%%%%%%%%%%%%%%%%%%%%%%%%%%%%%%%%%%%%%%%%%%%%%%%%%%%%%%
\section*{Acknowledgements}
Research partially supported by projects
MTM2012-30951,
Gen. Cat. DGR 2009SGR1040,
ESF EUROCORES programme EUROGIGA-ComPoSe IP04-MICINN,
MTM2011-28800-C02-01,
Gen. Cat. DGR 2009SGR1387

%%%%%%%%%%%%%%%%%%%%%%%%%%%%%%%%%%%%%%%%%%%%%%%%%%%%%%%
%%%%%%%%%%%%%%%%%%%%%%%%%%%%%%%%%%%%%%%%%%%%%%%%%%%%%%%

%%%%%%%%%%%%%%%%%%%%%%%%%%%%%%%%%%%%%%%%%%%%%%%%%%%%%%%
%%%%%%%%%%%%%%%%%%%%%%%%%%%%%%%%%%%%%%%%%%%%%%%%%%%%%%%
%%%%%%%%%%%%%%%%%%%%%%%%%%%%%%%%%%%%%%%%%%%%%%%%%%%%%%%
%%%%%%%%%%%%%%%%%%%%%%%%%%%%%%%%%%%%%%%%%%%%%%%%%%%%%%%

\end{document}